\documentclass[11pt]{article}%
\usepackage{amssymb}
\usepackage{eurosym}
\usepackage{amsfonts}
\usepackage{amsmath}
\usepackage{graphicx}
\usepackage{epstopdf}%
\providecommand{\U}[1]{\protect\rule{.1in}{.1in}}
\newtheorem{theorem}{Theorem}
\newtheorem{acknowledgement}[theorem]{Acknowledgement}

\newtheorem{example}[theorem]{Example}

\newtheorem{lemma}[theorem]{Lemma}

\newtheorem{proposition}[theorem]{Proposition}

\newenvironment{proof}[1][Proof]{\noindent\textbf{#1.} }{\ \rule{0.5em}{0.5em}}
\begin{document}

\title{Recurrent Trajectories and Finite Critical Trajectories of Quadratic
Differentials on the Riemann Sphere}
\author{Faouzi Thabet\\University of Gabes, Tunisia}
\maketitle

\begin{abstract}
In this paper, the focus will be on both the existence and non-existence
respectively of finite critical trajectories and recurrent trajectories of a
quadratic differential on the Riemann sphere. We show the connection between
these two items. More precisely, we collect some criterions for the
non-existence of recurrent trajectories. The first criterion is in the same
vein of Jenkins Three-pole Theorem, while the second one is in relation with
the so-called Level function of quadratic differentials.

\end{abstract}

\bigskip\textit{2010 Mathematics subject classification: 30C15, 31A35, 34E05.}

\textit{Keywords and phrases: Quadratic differentials. Horizontal and vertical
trajectories. Recurrent trajectories. Geodesics.}

\section{\bigskip Introduction}

\bigskip One of the most frequent problems once investigating the finite
critical trajectories of a given quadratic differential on the Riemann sphere
$\widehat{\mathcal{%
\mathbb{C}
}}$ is the existence of the infamous recurrent trajectories. Jenkins
Three-pole Theorem paves the way in a particular case. In this paper, our main
emphasis is on the connection between the existence of these two kinds of
trajectories, which appears to be permanent. More precisely, we gather two
criterions for the non-existence of recurrent trajectories. The first
criterion is based on a proof of Jenkins Three-pole Theorem. The second one is
derived from the so-called Level function of quadratic differentials as
determined by Y.Baryshnikov and B.Shapiro.

A quadratic differential on the Riemann sphere $\widehat{\mathcal{%
\mathbb{C}
}}$ is a $2$-form $\varpi_{\varphi}\left(  z\right)  =\varphi\left(  z\right)
dz^{2}$ where $\varphi$ is a rational function on $%
\mathbb{C}
.$ In a first plan, we give some immediate and brief observations from the
theory of quadratic differentials. For more details, we refer the reader to
\cite{Strebel},\cite{jenkins}...

\emph{Critical points }of $\varpi_{\varphi}$ are its zero's and poles in
$\widehat{\mathcal{%
\mathbb{C}
}}.$ Zeros and simple poles are called \emph{finite critical points, }while
poles of order $2$ or greater are called \emph{infinite critical points. }All
other points of $\widehat{\mathcal{%
\mathbb{C}
}}$ are called \emph{regular points}.

\emph{Horizontal trajectories} (or just trajectories) of the quadratic
differential $\varpi_{\varphi}$ are the zero loci of the equation%
\begin{equation}
\mathcal{\Im}\int^{z}\sqrt{\varphi\left(  t\right)  }\,dt=\text{\emph{const}},
\label{eq traj}%
\end{equation}
or equivalently%
\[
\varpi_{\varphi}=\varphi\left(  z\right)  dz^{2}>0.
\]

The \emph{vertical} (or, \emph{orthogonal}) trajectories are obtained by
replacing $\Im$ by $\Re$ in equation (\ref{eq traj}). Horizontal and vertical
trajectories of the quadratic differential $\varpi_{\varphi}$ produce two
pairwise orthogonal foliations of the Riemann sphere $\widehat{\mathcal{%
\mathbb{C}
}}$.

A trajectory passing through a critical point of $\varpi_{\varphi}$ is called
\emph{critical trajectory}. In particular, if it starts and ends at a finite
critical point, it is called \emph{finite critical trajectory }or\emph{\ short
trajectory}, otherwise, we call it an \emph{infinite critical trajectory}. The
closure of the set of finite and infinite critical trajectories, that we
denote by $\Gamma_{\varphi},$ is called the \emph{critical graph}.

The local and global structure of the trajectories are studied in
\cite[Theorem3.5]{jenkins},\cite{Strebel}. In the large, any trajectory is
either is a closed curve containing no critical points, or, is an arc
connecting two critical points, or, is an arc that has no limit along at least
one of its rays. The structure of the set $\widehat{%
\mathbb{C}
}\setminus\Gamma_{\varphi}$ depends on the local and global behaviors of
trajectories. It consists of a finite number of domains called the
\emph{domain configurations} of $\varpi_{\varphi}.$ The following Theorem
gives the possible structures of domain configurations (see \cite[Theorem3.5]%
{jenkins},\cite{Strebel}):

\begin{theorem}
\label{jenkins th}\bigskip There are five kinds of domain configurations, :

\begin{itemize}
\item \emph{Half-plane} \emph{domain}: It is swept by trajectories converging
to a pole of order 3 in its two ends, and along consecutive critical
directions. Its conformally mapped to a half plane.

\item \emph{Strip domain}: It is swept by trajectories which both ends tend to
poles of order 2. Its conformally mapped to a strip domain.

\item \emph{Ring domain}: It is swept by closed trajectories. Its conformally
mapped to an annulus.

\item \emph{Circle domain} : It is swept by closed trajectories and contains
exactly one double pole. Its conformally mapped to the a circle.

\item \emph{Dense domain : }It is swept by recurrent critical trajectory,
i.e., the interior of its closure is non-empty.
\end{itemize}
\end{theorem}

As it is mentioned in the title, we are interested in the fifth kind of domain
configuration. The dense domain is either the whole complex plane, or, is
bordered by finite critical trajectories. Figure \ref{FIG1} illustrates two
quadratics differentials with recurrent trajectories. In the first one, the
whole complex plane is a set domain, while in the second, the dense domain is
bounded by a closed finite critical trajectory. See \cite{kaplan},\cite{martz
rakh}.

\begin{figure}[tbh]
\begin{minipage}[b]{0.42\linewidth}
\centering\includegraphics[scale=0.4]{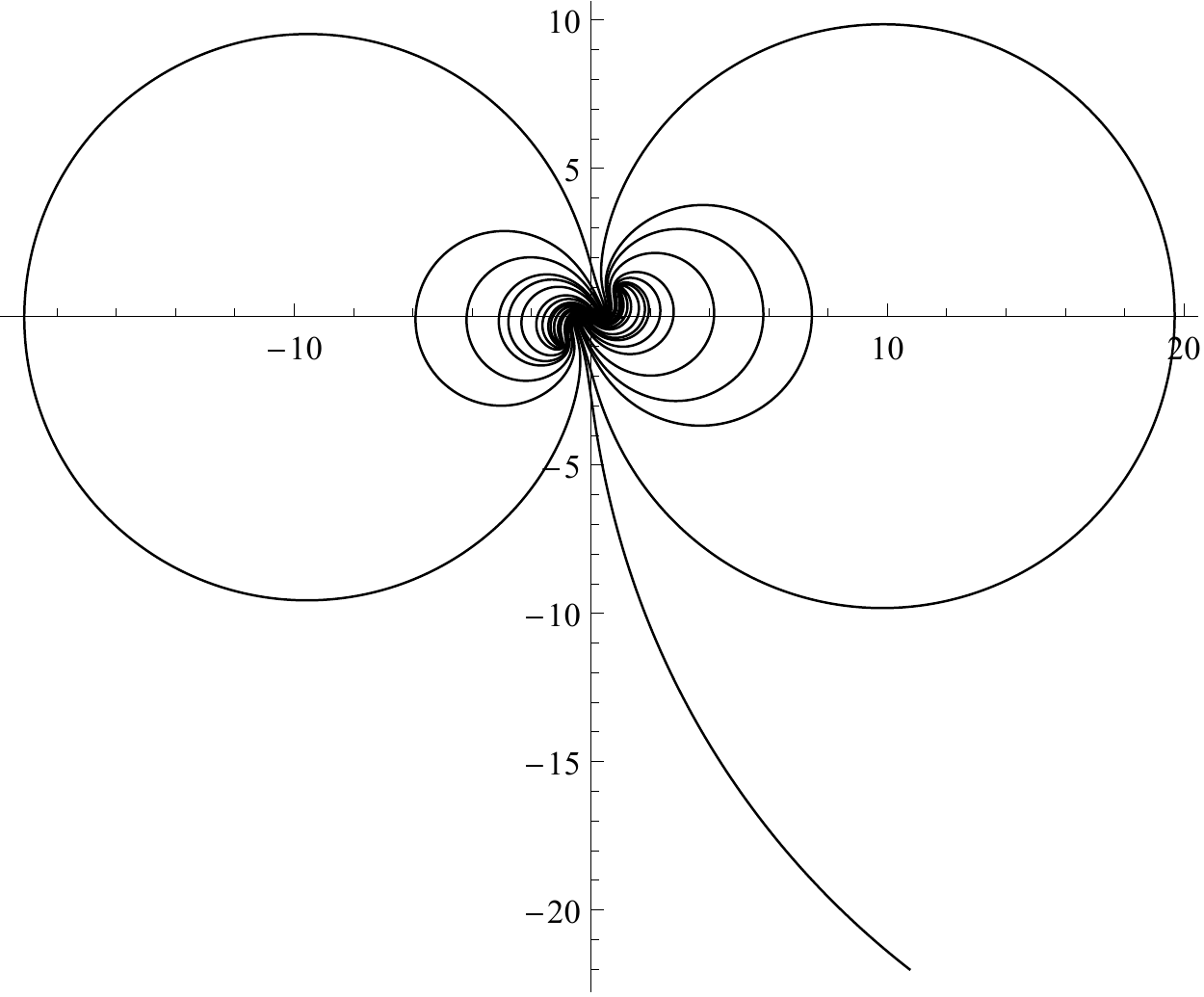}
\end{minipage}\hfill
\begin{minipage}[b]{0.32\linewidth} \includegraphics[scale=0.3]{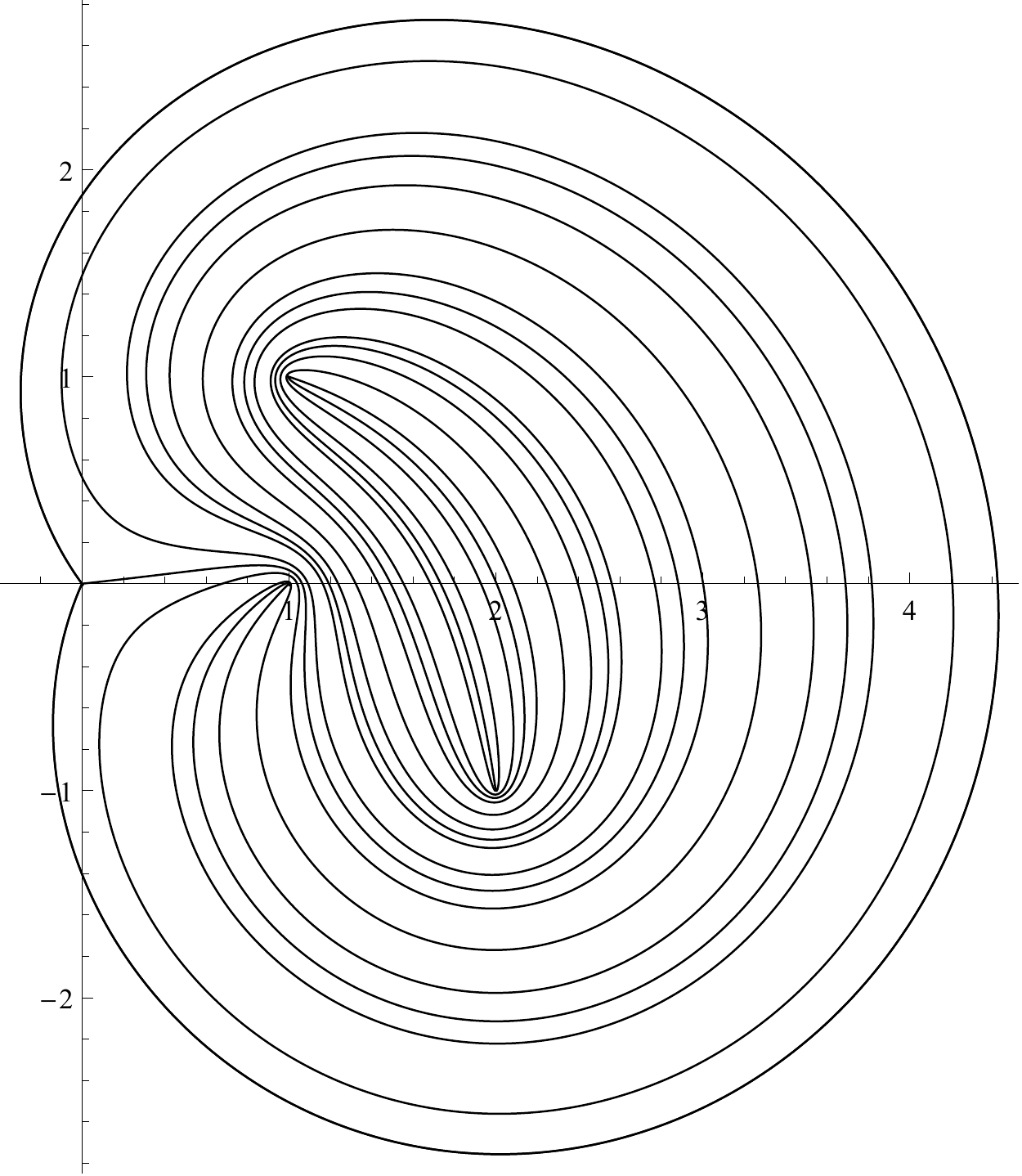}
\end{minipage}
\hfill\caption{Recurrent trajectory through the regular point $z=1$ of the
quadratic differential $-dz^{2}/\left(  z-0.5\right)  \left(  z+0.5\right)
\left(  z-1-i\right)  \left(  z+1+i\right)  ,$ (left). $-zdz^{2}/\left(
z-0.5\right)  \left(  z-1-i\right)  \left(  z-2+i\right)  $ (right)}%
\label{FIG1}%
\end{figure}

The quadratic differential $\varpi_{\varphi}$ defines a $\varphi$-metric on
the Riemann sphere with the differential element $\sqrt{\left\vert
\varphi\left(  z\right)  \right\vert }\left\vert dz\right\vert $. If $\gamma$
is a rectifiable arc in $\widehat{\mathcal{%
\mathbb{C}
}}$, then its $\varphi$-length is defined by%
\[
\left\vert \gamma\right\vert _{\varphi}=\int_{\gamma}\sqrt{\left\vert
\varphi\left(  z\right)  \right\vert }\left\vert dz\right\vert .
\]

A trajectory of $\varpi_{\varphi}$ is finite if, and only if its $\varphi
$-length is finite, otherwise it is infinite.

\bigskip A domain in $\widehat{\mathcal{%
\mathbb{C}
}}$ bounded only by segments of horizontal and/or vertical trajectories of
$\varpi_{\varphi}$ (and their endpoints) is called $\varpi_{\varphi}$-polygon.
A usefull tool in the investigation of the critical graph of a quadratic
differential is :

\begin{lemma}
[Teichm\H{u}ller]\label{teich lemma} Let $\Omega$ be a $\varpi_{\varphi}%
$-polygon, and let $z_{j}$ be the critical points on the boundary
$\partial\Omega$ of $\Omega,$ and let $t_{j}$ be the corresponding interior
angles with vertices at $z_{j},$ respectively. Then%
\begin{equation}
\sum_{j}\left(  1-\dfrac{\left(  n_{j}+2\right)  t_{j}}{2\pi}\right)
=2+\sum_{i}m_{i}, \label{Teich equality}%
\end{equation}
where $n_{j}$ are the multiplicities of $z_{j}=1,$ and $m_{i}$ are the
multiplicities of critical points inside $\Omega.$
\end{lemma}

\section{Recurrent trajectories}

Suppose that the quadratic differential $\varpi_{\varphi}$ has an infinite
critical point $p.$ From the local behavior of the trajectories, there exists
a neighborhood $\mathcal{U}$ of $p,$ such that, any trajectory in
$\mathcal{U}$, either is a closed curve encircling $p,$ if $p$ is a double
pole with negative residue, or, it diverges to $p$ following a certain
direction depending on the order of $p.$ In other word, there is absence of
recurrent trajectory in $\mathcal{U}.$ If $\varpi_{\varphi}$ has a recurrent
trajectory, then its respective dense domain cannot be the whole complex
plane, and then, $\varpi_{\varphi}$ should have finite critical trajectories
(Theorem \ref{jenkins th}). These facts show the following :

\begin{proposition}
\bigskip Assume that the quadratic differential $\varpi_{\varphi}$ satisfies :

\begin{enumerate}
\item[(i)] it has at least an infinite critical point,

\item[(ii)] it has no finite critical trajectory.
\end{enumerate}

Then, $\varpi_{\varphi}$ has no recurrent trajectory.
\end{proposition}

In general, the non-existence of recurrent trajectories is not guaranteed. The
most well-known result is the so-called Jenkins Three-pole Theorem, which
asserts that

\begin{theorem}
[Jenkins Three-pole Theorem]A quadratic differential on the Riemann sphere
with at most three distinct poles cannot have recurrent trajectories.
\end{theorem}

Two different proofs of this theorem can be found in \cite{jenkins} and
\cite[Theorem 15.2]{Strebel}. But the proof given in the second reference
seems to get a more precise result. What comes next is a refinement of the
mentioned Theorem :

\begin{proposition}
\bigskip\label{th princ}No recurrent trajectories of a quadratic differential
on the Riemann sphere with at most three critical points with odd multiplicities.
\end{proposition}

\begin{proof}
\bigskip We follow the idea of the proof of \cite[Theorem 15.2]{Strebel}. Let
$\gamma\left(  t\right)  ,t\in\left[  0,1\right[  ,$ be a ray of a recurrent
trajectory of a quadratic differential $\varpi_{\varphi}$ on the Riemann
sphere with at most three critical points with odd multiplicities. Set
$t_{0}\in\left[  0,1\right[  $ such that $\gamma\left(  t\right)  ,t\geq
t_{0}$ does not pass through any critical point. Let $\gamma^{\perp}$ be an
interval of the orthogonal trajectory starting at the point $z_{0}%
=\gamma\left(  t_{0}\right)  .$ By hypothesis, $\gamma\left(  t\right)  ,t\geq
t_{0}$ cuts $\gamma^{\perp}$ infinitely many times. Let $z_{1}=\gamma\left(
t_{1}\right)  ,t_{1}>t_{0}$ be the first intersection between $\gamma\left(
t\right)  $ and $\gamma^{\perp}.$ We consider the Jordan curve $\delta$
composed by the two sub-arcs $\gamma_{0}$ and $\gamma_{0}^{\perp}$
respectively of $\gamma\left(  t\right)  ,t\geq t_{0}$ and $\gamma^{\perp}$
between the points $z_{0}$ and $z_{1},$ see Figure \ref{FIG2}. As we can take
the next intersection between $\gamma\left(  t\right)  ,t>t_{1}$ and
$\gamma^{\perp},$ we may assume that the interior of $\delta$ at $z_{1}$
equals $\frac{3\pi}{2}.$ Let $z_{2}=\gamma\left(  t_{2}\right)  ,t_{2}>t_{1}$
be the next intersection of $\gamma\left(  t\right)  ,t>t_{1}$ and
$\gamma^{\perp}.$ Clearly, $z_{2}\in\gamma_{0}^{\perp},$ the sub-arc
$\gamma\left(  \left[  t_{1},t_{2}\right]  \right)  $ splits a component
$\Omega$ of $\widehat{%
\mathbb{C}
}\setminus\delta$ into two sub-domains $\Omega_{1}$ and $\Omega_{2},$ with
\begin{align*}
\partial\Omega_{1}  &  =\gamma\left(  \left[  t_{0},t_{2}\right]  \right)
\cup\gamma_{1}^{\perp},\\
\partial\Omega_{2}  &  =\gamma\left(  \left[  t_{1},t_{2}\right]  \right)
\cup\gamma_{2}^{\perp},
\end{align*}
where $\gamma_{1}^{\perp}$ and $\gamma_{2}^{\perp}$ are the sub-arcs of
$\gamma_{0}^{\perp}$ joining $z_{2}$ respectively to $z_{0}$ and $z_{1}.$
\[%
\begin{tabular}
[c]{|l|l|l|}\hline
& vertices & respective interior angles\\\hline
$\partial\Omega_{1}$ & $z_{0},z_{1},z_{2}$ & $\pi/2,\pi,\pi/2$\\\hline
$\partial\Omega_{2}$ & $z_{1},z_{2}$ & $\pi/2,\pi/2$\\\hline
\end{tabular}
\ \ \ \ \ \ \ \ \ \ \
\]
$\ $Applying Lemma \ref{teich lemma} to the $\varpi_{\varphi}$-polygons
$\partial\Omega_{1}$ and $\partial\Omega_{2},$ we get%
\[
1=2+\sum m_{i}\text{ for }\partial\Omega_{i},i=1,2.
\]
We conclude that each of the domains $\Omega_{1}$ and $\Omega_{2}$ must
contain at least a critical point with odd multiplicity. Now, let
$z_{3}=\gamma\left(  t_{3}\right)  ,t_{3}>t_{2}$ be the next intersection of
$\gamma\left(  t\right)  ,t>t_{2}$ and $\gamma^{\perp}.$ It is obvious that
the union of $\gamma\left(  \left[  t_{2},t_{3}\right]  \right)  $ and the
sub-arc of $\gamma_{0}^{\perp}$ between $z_{2}$ and $z_{3}$ splits the
exterior of $\Omega$ into two sub-domains $\digamma_{1}$ and $\digamma_{2}.$%
\[%
\begin{tabular}
[c]{|l|l|l|}\hline
& vertices & respective interior angles\\\hline
$\partial\digamma_{1}$ & $z_{0},z_{1},z_{2},z_{3}$ & $3\pi/2,\pi/2,\pi
/2,\pi/2$\\\hline
$\partial\digamma_{2}$ & $z_{2},z_{3}$ & $\pi/2,\pi/2$\\\hline
\end{tabular}
\ \ \ \ \ \ \ \ \ \ \
\]
Once more, applying Lemma \ref{teich lemma} to the $\varpi_{\varphi}$-polygons
$\partial\digamma_{1}$ and $\partial\digamma_{2},$ we reach the same
conclusion that each of the domains $\digamma_{1}$ and $\digamma_{2}$ must
contain at least a critical point with odd multiplicity. We just proved that
the quadratic differential should have at least four critical points with odd
multiplicities; a contradiction. \begin{figure}[th]
\centering\includegraphics[height=1.5in,width=2in]{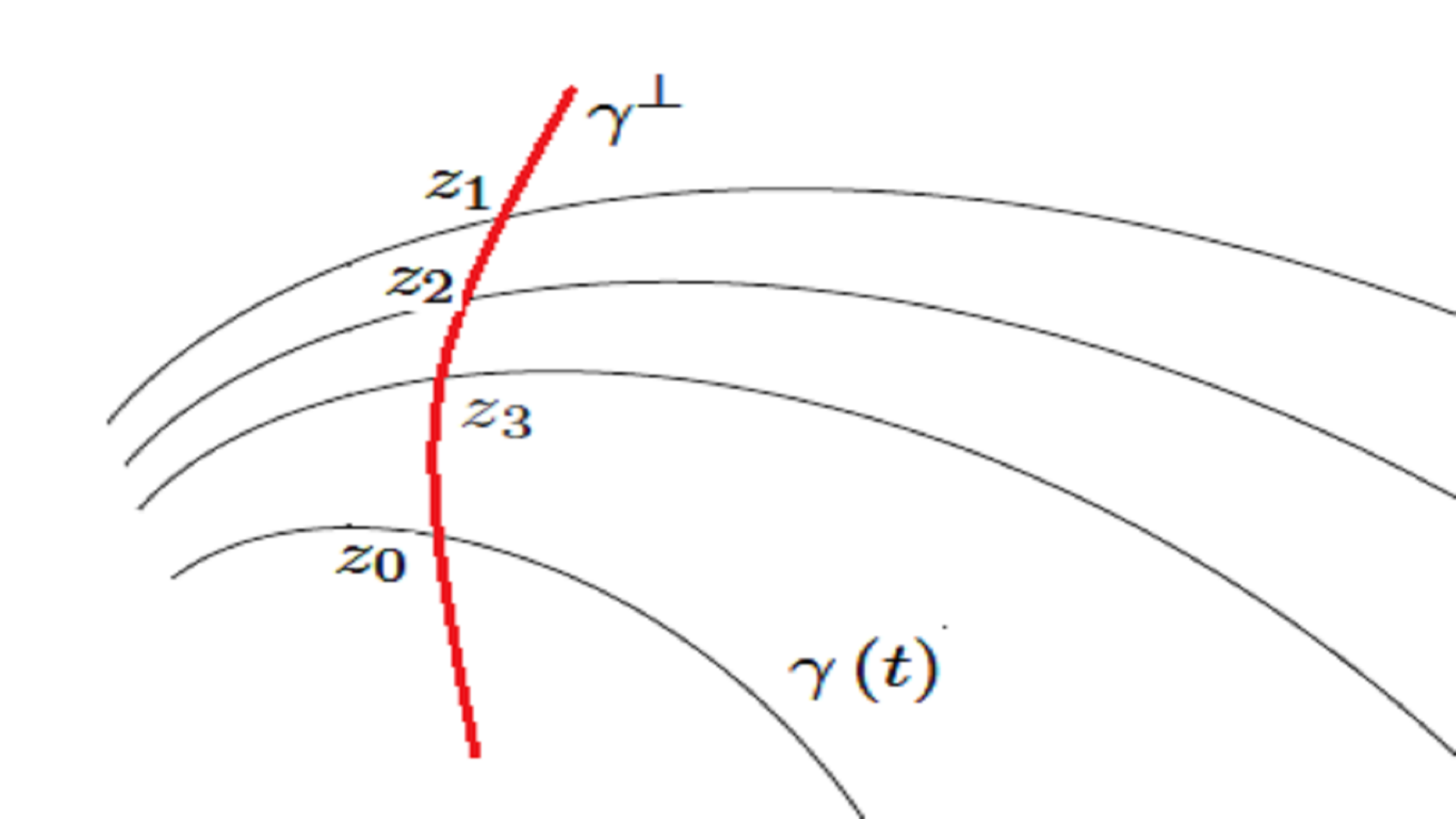} \hfill\caption{}%
\label{FIG2}%
\end{figure}
\end{proof}

\bigskip

\begin{example}
A particular case is when all the multiplicities of the critical points of a
quadratic differential are even. A lemniscate of a rational function $r\left(
z\right)  =\frac{p\left(  z\right)  }{q\left(  z\right)  },$ where $p$ and $q$
are two co-prime polynomials, is defined for $c>0$ by
\[
\left\{  z\in\widehat{%
\mathbb{C}
}:\left\vert r\left(  z\right)  \right\vert =c\right\}  .
\]
It is a real algebraic curve of degree $2\max(\deg p,\deg q);$ indeed, its
defining equation can be seen as
\[
p(x,y)\overline{p(x,y)}-c^{2}q(x,y)\overline{q(x,y)}=0,
\]
with $z=x+iy.$ For more details, see \cite{Sheil-Small}.These sets can be
determined immediately (equation (\ref{eq traj})) as trajectories of the
quadratic differential
\[
-\left(  \frac{r^{\prime}\left(  z\right)  }{r\left(  z\right)  }\right)
^{2}dz^{2}=-\left(  \frac{p^{\prime}\left(  z\right)  q\left(  z\right)
-p\left(  z\right)  q^{\prime}\left(  z\right)  }{p\left(  z\right)  q\left(
z\right)  }\right)  ^{2}dz^{2}.
\]
From the equality
\[
\frac{p^{\prime}\left(  z\right)  q\left(  z\right)  -p\left(  z\right)
q^{\prime}\left(  z\right)  }{p\left(  z\right)  q\left(  z\right)  }%
=\sum_{p\left(  a\right)  q\left(  a\right)  =0}\frac{m_{a}}{z-a},
\]
where $m_{a}\in%
\mathbb{N}
^{\ast}$ are the multiplicity of the zero's $a$ of $p\left(  z\right)
q\left(  z\right)  ,$ we deduce that the finite critical points are the zero's
of $\sum_{p\left(  a\right)  q\left(  a\right)  =0}\frac{m_{a}}{z-a};$ the
zero's of $p\left(  z\right)  q\left(  z\right)  $ (and possibly $\infty$) are
the only infinite critical points of $-\left(  \frac{r^{\prime}\left(
z\right)  }{r\left(  z\right)  }\right)  ^{2}dz^{2},$ and they are all double
poles with negative residues. By Proposition \ref{th princ}, this quadratic
differential has no recurrent trajectories. Its critical graph can be found
easily since the local and global structures of the trajectories are
well-known. The set of closed trajectories of $-\left(  \frac{r^{\prime
}\left(  z\right)  }{r\left(  z\right)  }\right)  ^{2}dz^{2}$ cover the whole
complex plane minus a zero Lebegue set, this kind of quadratic differential is
called \emph{Strebel differential}.
\end{example}

\bigskip A more elaborated exploitation of the proof of Proposition
\ref{th princ} allows us to show the following :

\begin{proposition}
\bigskip Let $\varphi\left(  z\right)  =\frac{p\left(  z\right)  }{\left(
q\left(  z\right)  \right)  ^{2}},$ where $p\left(  z\right)  $ and $q\left(
z\right)  $ are two co-prime polynomials. We suppose that :

\begin{enumerate}
\item[(a)] for each $i=1,...,n,$ the zero's $a_{2i-1}$ and $a_{2i}$ of
$p\left(  z\right)  $ are connected by a short trajectory of $\varpi_{\varphi
},$

\item[(b)] each pair of zero's $a_{2i-1}$ and $a_{2i}$ have the same parity of multiplicities,
\end{enumerate}

Then, the quadratic differential $\varpi_{\varphi}$ has no recurrent trajectory.
\end{proposition}

\bigskip

Here is a starting point which is an idea of Y.Baryshnikov and B.Shapiro, it
gives a necessary and sufficient conditions on the existence of recurrent
trajectories of the quadratic differential $\varpi_{\varphi}$ :

\begin{lemma}
[\cite{shabary}]\label{sha barysh}Assume that there exists a function $f:%
\mathbb{C}
\setminus\mathcal{I}\longrightarrow%
\mathbb{R}
$ (called \emph{level function} of $\varpi_{\varphi}$) such that :

\begin{enumerate}
\item[(i)] $f$ is continuous and piecewise smooth on $%
\mathbb{C}
\setminus\mathcal{I};$

\item[(ii)] $f$ is constant on the trajectories of $\varpi_{\varphi};$

\item[(iii)] $f$ is non-constant on any open subset of $%
\mathbb{C}
.$
\end{enumerate}

Then, the quadratic differential $\varpi_{\varphi}$ has no recurrent
trajectory. Conversely, if $\varpi_{\varphi}$ has no recurrent trajectory,
then such a function exists.
\end{lemma}

\begin{proof}
If $\varpi_{\varphi}$ has no recurrent trajectory, $\Gamma_{\varphi}$ splits
$\widehat{%
\mathbb{C}
}$ into at most the first four domain configurations defined previously in the
introduction : half-plane domains, ring domains, circle domains, and strip
domains. On each of these domains we can construct a function that is
continuous, constant on the trajectories, but not on any open set, and which
is vanishing on the boundary of the domain. Gluing together these functions
along delivers the desired continuous function.

If $\varpi_{\varphi}$ has a recurrent trajectory $\gamma,$ by continuity, the
function $f$ will be constant on its closure, which violates (iii) of.
\end{proof}

This leads us to an interesting particular case when $\varphi\left(  z\right)
=\frac{p\left(  z\right)  }{\left(  q\left(  z\right)  \right)  ^{2}},$ where
$p\left(  z\right)  $ and $q\left(  z\right)  $ are two co-prime polynomials :

\begin{proposition}
\bigskip\label{main1}We suppose that :

\begin{enumerate}
\item[(a)] for each $i=1,...,n,$ the zero's $a_{2i-1}$ and $a_{2i}$ of
$p\left(  z\right)  $ are connected by a short trajectory of $\varpi_{\varphi
},$

\item[(b)] the residue of the rational function $\varphi\left(  z\right)  $ in
every zero of $q\left(  z\right)  $ is purely imaginary number.
\end{enumerate}

Then, the quadratic differential $\varpi_{\varphi}$ has no recurrent trajectory.
\end{proposition}

\begin{proof}
For $i=1,...,n,$ let $\gamma_{i}$ be the short trajectory (taken with an
orientation) of $\varpi_{\varphi}$ connecting $a_{2i-1}$ and $a_{2i}.$ We
define a single-valued function $\sqrt{p\left(  z\right)  }$ in $%
\mathbb{C}
\setminus\cup_{i=1}^{n}\gamma_{i}.$ \bigskip For $t\in\cup_{i=1}^{n}\gamma
_{i},$ we denote by $\left(  \sqrt{p\left(  t\right)  }\right)  _{+}$ and
$\left(  \sqrt{p\left(  t\right)  }\right)  _{-}$the limits of $\sqrt{p\left(
z\right)  }$ from the $+$ and $-$ sides, respectively. (As usual, the $+$-side
of an oriented curve lies to the left and the $-$-side lies to the right, if
one traverses the curve according to its orientation). We have then%
\[
\frac{\left(  \sqrt{p\left(  t\right)  }\right)  _{+}}{q\left(  t\right)
}=-\frac{\left(  \sqrt{p\left(  t\right)  }\right)  _{-}}{q\left(  t\right)
},t\in\gamma_{i},i=1,...,n.
\]
Since $\gamma_{i}$ is a short trajectory of $\varpi_{\varphi},$
\[
\Im\left(  \int_{a_{2i-1}}^{t}\frac{\left(  \sqrt{p\left(  s\right)  }\right)
_{+}}{q\left(  s\right)  }ds\right)  =0,t\in\gamma_{i},i=1,...,n.
\]
It follows that, if $\gamma$ is a closed Jordan curve encircling $\gamma_{i}$
(or parts of it ) and none of the poles of $\varphi,$ then
\[
\Im\int_{\gamma}\frac{\sqrt{p\left(  z\right)  }}{q\left(  z\right)  }%
dz=\pm2\Im\int_{\gamma_{i}}\frac{\left(  \sqrt{p\left(  z\right)  }\right)
_{+}}{q\left(  z\right)  }dz=0.
\]
We consider the multi-valued function
\[
f:z\longmapsto\int_{a_{1}}^{z}\frac{\sqrt{p\left(  s\right)  }}{q\left(
s\right)  }ds
\]
defined in $%
\mathbb{C}
\setminus\mathcal{I}.$ For $z\in%
\mathbb{C}
\setminus\left(  \left(  \cup_{i=1}^{n}\gamma_{i}\right)  \cup\mathcal{I}%
\right)  ,$ let $\beta_{1}$ and $\beta_{2}$ be two oriented Jordan arcs
connecting $a_{1}$ and $z$ in $%
\mathbb{C}
\setminus\left(  \left(  \cup_{i=1}^{n}\gamma_{i}\right)  \cup\mathcal{I}%
\right)  .$ If we denote by $\mathcal{J}$ the subset of poles of $\varphi$
encircled by the Jordan closed curve $\beta_{1}-\beta_{2},$ then we have
\[
\int_{\beta_{1}-\beta_{2}}\frac{\sqrt{p\left(  s\right)  }}{q\left(  s\right)
}ds=\pm2\pi i\sum_{a\in\mathcal{J}}res_{f}\left(  a\right)  .
\]
By condition (a), we conclude that
\[
\Im\int_{\beta_{1}}\frac{\sqrt{p\left(  s\right)  }}{q\left(  s\right)
}ds=\Im\int_{\beta_{2}}\frac{\sqrt{p\left(  s\right)  }}{q\left(  s\right)
}ds,
\]
and the function $\Im f$ is well defined in $%
\mathbb{C}
\setminus\left\{  \text{poles of }\varphi\right\}  .$ Obviously, it satisfies
(i) and (ii) of Lemma \ref{sha barysh}; the point (iii) follows from the
harmonicity of $\Im f$ in $%
\mathbb{C}
\setminus\left(  \left(  \cup_{i=1}^{n}\gamma_{i}\right)  \cup\left\{
\text{poles of }\varphi\right\}  \right)  .$
\end{proof}

As a consequence, if $p$ and $q$ have simple and real zero's respectively
$a_{1}<\cdot\cdot\cdot<a_{2n}$ and $b_{1}<\cdot\cdot\cdot<b_{m},$ such that
$p\left(  b_{i}\right)  <0$ for $i=1,...,m.$ Then the quadratic differential
$-\frac{p\left(  z\right)  }{\left(  q\left(  z\right)  \right)  ^{2}}dz^{2}$
has no recurrent trajectory. Indeed, for any $i=1,...,n,$ the segment $\left[
a_{2i-1},a_{2i}\right]  $ is a short trajectory of the quadratic differential
$-\frac{p\left(  z\right)  }{\left(  q\left(  z\right)  \right)  ^{2}}dz^{2},$
and the result follows from Proposition \ref{main1}.

\begin{example}
Let $p\left(  z\right)  ,q\left(  z\right)  ,$ and $r\left(  z\right)  $ be
three complex polynomials, $p$ and $q$ are co-prime. If the algebraic
equation
\[
p\left(  z\right)  \mathcal{C}^{2}\left(  z\right)  +q\left(  z\right)
\mathcal{C}\left(  z\right)  +r\left(  z\right)  =0
\]
admits solution as Cauchy transform of some compactly supported Borel signed
measure $\mu$ in the $%
\mathbb{C}
$-plane, then, by Plemelj-Sokhotsky Lemma, this measure lives in the short
trajectories of the quadratic differential
\[
-\frac{q^{2}\left(  z\right)  -4p\left(  z\right)  r\left(  z\right)  }%
{p^{2}\left(  z\right)  }dz^{2},
\]
and it is given by
\[
d\mu\left(  z\right)  =\frac{1}{2\pi i}\frac{\sqrt{q^{2}\left(  z\right)
-4p\left(  z\right)  r\left(  z\right)  }}{p\left(  z\right)  }dz,
\]
see \cite{pristker} or \cite{Shapiro}. The information by Proposition
\ref{main1} that the above quadratic differential cannot have a recurrent
trajectory is essential in the finding of short trajectories if exist.
\end{example}

\begin{acknowledgement}
\bigskip\ This work was carried out during a visit in July 2017 to the Q.G.M
centre, Aahr\"{u}s University, Denmark. The author wants to thank Professor
J\o rgen Ellegaard Anderson for the invitation, and the whole staff of the
centre for their kindness and hospitality.
\end{acknowledgement}


\begin{verbatim}
Higher institute of applied sciences and technology of Gabes,
Avenue Omar Ibn El Khattab 6029. Gabes. Tunisia.
E-mail address: faouzithabet@yahoo.fr
\end{verbatim}

\end{document}